\documentclass[10pt]{amsart}

\usepackage[mathcal]{euscript}
\usepackage{amssymb, amsfonts, xypic, amsmath}
\usepackage{amscd}
\usepackage[all]{xy}
\usepackage{dutchcal}
\usepackage{amsthm}
\usepackage[margin=1.5in]{geometry}
\usepackage{xcolor}
\usepackage{tikz-cd}

\newtheorem{theorem}{Theorem}

\newtheorem{proposition}[theorem]{Proposition}

\theoremstyle{definition}
\newtheorem{definition}[theorem]{Definition}

\newtheorem{remark}{Remark}
\newtheorem*{theorem*}{Theorem}

\numberwithin{equation}{section} \numberwithin{figure}{section}

\numberwithin{equation}{section}
\DeclareRobustCommand{\rchi}{{\mathpalette\irchi\relax}}
\newcommand{\irchi}[2]{\raisebox{\depth}{$#1\chi$}}

\newcommand{\Z}{\mathbb{Z}}

\newtheorem*{ack}{Acknowledgments}

\usepackage{graphicx}
\usepackage{epstopdf}

\author{Manuel Rivera, Felix Wierstra, Mahmoud Zeinalian}

\newcommand{\Addresses}{{% additional braces for segregating \footnotesize
  \bigskip
  \footnotesize

\textsc{Manuel Rivera, Department of Mathematics, University of Miami, 1365 Memorial Drive, Coral
Gables, FL 33146 and Departamento de Matem\'aticas, Cinvestav, Av. Instituto Polit\'ecnico Nacional 2508, Col. San Pedro Zacatenco, M\'exico, D.F. CP 07360, M\'exico}\par\nopagebreak \textit{E-mail address} \texttt{manuelr@math.miami.edu}
  
    \medskip
  \medskip

  \textsc{Felix Wierstra, Max Planck Institute for Mathematics, Vivatsgasse 7, 53111 Bonn, Germany}\par\nopagebreak
  \textit{E-mail address} \texttt{felix.wierstra@gmail.com}

  \medskip
  \medskip
  
  \textsc{Mahmoud Zeinalian, Department of Mathematics, City University of New York, Lehman College, 250 Bedford Park Blvd W, Bronx, NY 10468
   }\par\nopagebreak
  \textit{E-mail address} \texttt{mahmoud.zeinalian@lehman.cuny.edu}

}}

\begin{document}

\begin{abstract}
The normalized singular chains of a path connected pointed space $X$ may be considered as a connected $E_{\infty}$-coalgebra $\mathbf{C}_*(X)$ with the property that the $0^{\text{th}}$ homology of its cobar construction, which is naturally a cocommutative bialgebra, has an antipode, i.e. it is a cocommutative Hopf algebra. We prove that a continuous map of path connected pointed spaces $f: X\to Y$ is a weak homotopy equivalence if and only if $\mathbf{C}_*(f): \mathbf{C}_*(X)\to \mathbf{C}_*(Y)$ is an $\mathbf{\Omega}$-quasi-isomorphism, i.e. a quasi-isomorphism of dg algebras after applying the cobar functor $\mathbf{\Omega}$ to the underlying dg coassociative coalgebras. The proof is based on combining a classical theorem of Whitehead together with the observation that the fundamental group functor and the data of a local system over a space may be described functorially from the algebraic structure of the singular chains.
\end{abstract} 

\title[The functor of singular chains detects weak homotopy equivalences]{The functor of singular chains detects weak homotopy equivalences}
\maketitle
\section{Introduction}

The normalized singular chains of a path connected pointed space has a natural connected $E_{\infty}$-coalgebra structure extending the Alexander-Whitney coproduct. An explicit description of such structure may be found in \cite{McSm02} and \cite{BeFr04} in terms of an $E_{\infty}$-operad called the surjection operad and which is denoted by $\rchi$. The cobar construction of a $\rchi$-coalgebra, defined as the cobar construction of its underlying dg coassociative coalgebra, has a natural coproduct making it a dg bialgebra \cite{Ka03}, \cite{Fr10}. The cobar construction of the $\rchi$-coalgebra of normalized singular chains has a further property: the $0^{\text{th}}$ homology of its cobar construction is a Hopf algebra, i.e. it is a bialgebra with the property of admitting an antipode map (if a bialgebra admits an antipode then it is unique). We denote by $\mathbf{C}_*: Top^0 \to \mathcal{E}^0_{\mathbf{k}}$ the functor of normalized singular chains from the category $Top^0$ of semi-locally simply connected and locally path connected path connected pointed spaces to the category $\mathcal{E}^0_{\mathbf{k}}$ of connected $\rchi$-coalgebras with the property that the $0^{\text{th}}$ cobar homology, which is a bialgebra, has an antipode map making it into a Hopf algebra.
In this article we prove the following statement.
\begin{theorem*}
 A continuous map of path connected pointed spaces $f: X\to Y$ is a weak homotopy equivalence if and only if $\mathbf{C}_*(f): \mathbf{C}_*(X)\to \mathbf{C}_*(Y)$ is an $\mathbf{\Omega}$-quasi-isomorphism, i.e. a quasi-isomorphism of dg algebras after applying the cobar functor $\mathbf{\Omega}$ to the underlying dg coassociative coalgebras.
\end{theorem*}

The notion of $\mathbf{\Omega}$-quasi-isomorphism is stronger than that of an ordinary quasi-isomorphism. In other words, every $\mathbf{\Omega}$-quasi-isomorphism is a quasi-isomorphism, but not vice versa. 

In \cite{RiZe18} it is shown that $\mathbf{C}_*$ sends weak homotopy equivalences of path connected pointed spaces to $\mathbf{\Omega}$-quasi-isomorphisms. In this paper, after briefly reviewing the relevant background material, we prove the converse. The main ingredients of the proof are as follows: 
\\

(1) Obtaining the fundamental group from the singular chains. The noteworthy phenomenon is that the $\mathbf{\Omega}$-quasi-isomorphism class of $\mathbf{C}_*(X)$ determines the isomorphism class of the fundamental group of the space. This is the content of Section 3. In fact, we do not require the full $E_\infty$-coalgebra structure of the singular chains but only part of it: those structure maps that give rise to a bialgebra structure on the cobar construction of the underlying coassociative dg coalgebra of chains with Alexander-Whitney coproduct. In \cite{GeVo95} it is shown how this bialgebra structure on the cobar construction is coming from the $E_2$-coalgebra structure on the singular chains. In \cite{Ka03} this was extended to the full $E_\infty$-coalgebra structure and it is explained how this bialgebra structure on the cobar construction fits as part of the $E_{\infty}$-coalgebra structure of the chains. 
\\

(2) It was shown in \cite{RiZe17} that the homotopy colimit of an infinity local system of chain complexes over a space can be computed using the twisted tensor product of the chains on the base and the fibre. This chain complex is also quasi-isomorphic to the one sided bar construction of the dg algebra $A= \mathbf{\Omega} C$ (cobar construction on the dg coalgebra of singular chains $C$) and the $A$-module of chains on the fiber. In particular, when the infinity local system is an ordinary classical local system, the homology with local coefficients can be calculated as a one sided bar construction. Since the one sided bar construction is  quasi-isomorphism invariant, it follows that an $\mathbf{\Omega}$-quasi-isomorphism induces an isomorphism of homology with local coefficients.  This is discussed in Section 4.
\\

(3) Whitehead's Theorem says that a map of spaces that induces an isomorphism between fundamental groups and between homology with coefficients in the fundamental group ring, via the left action of the fundamental group on itself, is a weak homotopy equivalence. In Section 5 we review this version of Whitehead's Theorem and in Section 6 we use it to prove our main result. 

\begin{ack} The second and third authors would like to thank the \textit{Max Planck Institute for Mathematics} in Bonn, where they met for the first time and where the seeds of this joint work with the first author were sown. The second author also acknowledges the support of grant GA CR  No. P201/12/G028.
\end{ack}

\section{Preliminaries}

\subsection{Conventions} Throughout the article $\mathbf{k}$ will be a fixed commutative ring with unit.  All tensor products are taken over $\mathbf{k}$. For any group $G$, $\mathbf{k}[G]$ denotes the corresponding group ring. All differentials will have degree $-1$. All differential graded (dg) algebras and coalgebras will be associative and coassociative, respectively. Algebras and coalgebras will be assumed to be unital and counital, respectively. A bialgebra is a $\mathbf{k}$-module $H$ equipped with a product $\mu: H \otimes H\to H$, a coproduct $\nabla: H \to H \otimes H$, a unit $\eta: \mathbf{k} \to H$ and a counit $\epsilon: H \to \mathbf{k}$ such that $\nabla$ is an algebra map. By a Hopf algebra $H$ we mean a bialgebra that admits an antipode map, namely, a map $s: H \to H$ such that $\mu \circ (s \otimes \text{id}_H) \circ \nabla = \eta \circ \epsilon = \mu \circ (\text{id}_H \otimes s) \circ \nabla$. 

\subsection{Simplicial sets and normalized chains} Let $Set$ be the category of sets and let $\mathbf{\Delta}$ be the category of non-empty finite ordinals $[0], [1], [2], ...$ with order-preserving maps. Denote by $Set_{\mathbf{\Delta}} = \text{Fun}(\mathbf{\Delta}^{op}, Set)$ the category of simplicial sets and by $dgCoalg_{\mathbf{k}}$ the category of dg coalgebras over $\mathbf{k}$. Recall the definition of the dg coalgebra of simplicial chains $(C'_*(S), \partial, \Delta)$. Let $C'_n(S)$ be the free $\mathbf{k}$-module generated by the set of $n$-simplices $S_n$ and let $\partial: C'_n(S) \to C'_{n-1}(S)$ be the map given by $\partial(\sigma)= \sum_{i=0}^n (-1)^iS(d_i)(\sigma)$, where $d_i: [n-1] \to [n]$ is the $i^{\text{th}}$ coface map in the category $\mathbf{\Delta}$. It follows that $(C'_*(S), \partial)$ is a chain complex. The Alexander-Whitney coproduct $\Delta: C'_n(S) \to \bigoplus_{p+q=n} C'_i(S) \otimes C'_j(S)$ is defined by the formula
$$\Delta(\sigma)= \bigoplus_{p+q=n} S(f_p)(\sigma) \otimes S(l_q)(\sigma),$$
where $f_p: [p] \to [p+q]$ is given by $f_p(i)=i$ and $l_q: [q] \to [p+q]$ is given by $l_q(i)=p+i$. The triple $(C'_*(S), \partial, \Delta)$ is a dg coalgebra, namely, $\Delta$ is coassociative and compatible with the differential $\partial$. The sub graded $\mathbf{k}$-module $D_*(S) \subseteq C'_*(S)$ generated by the degenerate simplices forms a subcomplex of $(C'_*(S), \partial)$. Denote the quotient complex by $C_*(S) := C'_*(S)/D_*(S)$. The coproduct $\Delta$ induces a well defined chain map $\Delta: C_*(S) \to C_*(S) \otimes C_*(S)$ and therefore $(C_*(S), \partial, \Delta)$ defines a dg coalgebra. This construction is natural with respect to maps of simplicial sets and consequently gives rise to a functor 
$$C_*: Set_{\mathbf{\Delta}} \to dgCoalg_{\mathbf{k}}$$
called the normalized simplicial chains functor. If $S$ is a simplicial set such that $S_0=\{x\}$ then $(C_*(S), \partial, \Delta)$ is a counital connected dg coalgebra, i.e. $C_0(S) \cong \mathbf{k}$ and the counit is given by the projection onto $C_0$. Therefore, the functor $C_*$ restricts to a functor
$$C^0_*: Set^0_{\mathbf{\Delta}} \to dgCoalg^0_{\mathbf{k}},$$
where $Set^0_{\mathbf{\Delta}}$ denotes the category of simplicial sets with a single $0$-simplex and $dgCoalg^0_{\mathbf{k}}$ the category of connected dg coalgebras. 

Let $Top$ be the category of semi-locally simply connected, locally path connected topological spaces. It is a classical result that every object in $Top$ has a universal cover. Denote by  $$\text{Sing}: Top \to Set_{\mathbf{\Delta}}$$ the functor that sends a space to the Kan complex of singular simplices. For simplicity, we will denote the composition of functors $C \circ \text{Sing}$ again by $C$, so if $X$ is a topological space $C_*(X)$ denotes the dg coalgebra of normalized singular chains on $X$. Now let $Top^0$ be the category of semi-locally simply connected, locally path connected, path connected pointed spaces. Throughout this article, by the word space we mean a semi-locally simply connected, locally path connected topological space.  

Given any $(X,x) \in Top^0$ define $\text{Sing}^0(X,x) \in Set^0_{\mathbf{\Delta}}$ to be the sub simplicial set of $\text{Sing}(X)$ whose set of $n$-simplices $\text{Sing}^0(X,x)_n$ consists of all continuous maps $\sigma: \Delta^n \to X$ that send all the vertices of $\Delta^n$ to the base point $x \in X$.  This gives rise to a functor 
$$\text{Sing}^0: Top^0 \to Set^0_{\mathbf{\Delta}}.$$ Again, for simplicity, we will denote $C^0_* \circ \text{Sing}^0$ by $C^0_*$, so if $(X,x)$ is a path connected pointed space $C^0_*(X,x)$ will denote the connected dg coalgebra of normalized singular chains with vertices at $x$.

\begin{remark}
For any path connected pointed space $(X,x)\in Top^0$ the inclusion map $\iota: \text{Sing}^0_*(X) \hookrightarrow \text{Sing}_*(X)$ is a weak equivalence of simplicial sets. Furthermore, since both $\text{Sing}^0_*(X)$ and $\text{Sing}_*(X)$ are Kan complexes, it follows from Proposition 17.2.8 of \cite{Rie14} that $\iota$ is a categorical equivalence of quasi-categories. See also Section 8 of \cite{RiZe16} for more details.
\end{remark}

\subsection{The $E_{\infty}$-coalgebra structure on the normalized chains}

For any simplicial set $S$ the Alexan\-der-Whitney coproduct $\Delta: C_*(S) \to C_*(S) \otimes C_*(S)$ is cocommutative up to an infinite hierarchy of coherent homotopies. This observation goes back to Steenrod and it has been made explicit in several articles in the literature by using the notion of $E_{\infty}$-coalgebras. We follow \cite{McSm02} and \cite{BeFr04} where a coaction of an $E_{\infty}$-operad $\rchi$ called the surjection operad is constructed on the normalized chains $C_*(S)$ extending the Alexander-Whitney dg coassociative coalgebra structure. 

 The surjection operad $\rchi$ is a dg operad defined as the sequence $\{ \rchi(1), \rchi(2), .... \}$ of dg $\mathbf{k}$-modules such that $\rchi(r)_d$ is generated by surjections of sets $u: \{1,...,r+d\} \to \{1,...,r\}$ which satisfy $u(i) \neq u(i+1)$. For the definition of the differential $\rchi(r)_d \to \rchi(r)_{d-1}$  we refer to \cite{BeFr04}. The surjection operad can be given a filtration $F_1\rchi \subset F_2\rchi \subset ... \subset \rchi$ where $F_n\rchi$ is quasi-isomorphic to the normalized singular chains on the (topological) little $n$-cubes operad; in particular, $F_1\rchi$ is the associative operad. Thus for any $\rchi$-(co)algebra $\mathbf{C}$ there is an underlying dg (co)associative (co)algebra $C$. The surjection operad $\rchi$ is an $E_{\infty}$-operad, this means that $\rchi$ is quasi-isomorphic to the commutative operad and the action of the symmetric group $\Sigma_r$ on each $\rchi(r)$ yields a projective $\Sigma_r$-module. The details of the following classical theorem are spelled out in Section 2 of  \cite{BeFr04}.
 
 \begin{theorem} \label{BeFr} \cite{BeFr04} Let $S$ be a simplicial set. There are natural chain maps $\rchi(r) \otimes C_*(S) \to C_*(S)^{\otimes r}$ which make the normalized chains $C_*(S)$ into a coalgebra over the surjection operad $\rchi$ such that the operation $u_0: C_*(S) \to C_*(S) \otimes C_*(S)$ associated to the identity surjection $(u_0: \{1,2\} \to\{1, 2\}) \in \rchi(2)_0$ is exactly the Alexander-Whitney coproduct. 
 \end{theorem}
 
It follows that the normalized chains functor $C_*: Set_{\mathbf{\Delta}} \to dgCoalg_{\mathbf{k}}$ lifts to a functor $Set_{\mathbf{\Delta}} \to \rchi \text{-}Coalg$, where $\rchi\text{-}Coalg$ denotes the category of $\rchi$-coalgebras. We will denote by
$$\mathbf{C}_*: Top^0 \to \rchi\text{-}Coalg^0$$
the functor that sends a path connected pointed space $(X,x)$ to the connected $\rchi$-coalgebra $\mathbf{C}_*(X,x)$ given by the $\rchi$-coaction provided by Theorem \ref{BeFr} on the normalized chains of the simplicial set $\text{Sing}^0(X,x) \in Set_{\mathbf{\Delta}}^0$. Thus, using the notation introduced previously, the underlying dg coalgebra of $\mathbf{C}_*(X,x)$ is $C^0_*(X,x)$. We will omit the base point and simply write $\mathbf{C}_*(X)=\mathbf{C}_*(X,x)$ and $C^0_*(X)$ for its underlying dg coalgebra. In general, we will omit the base point from the notation when working with objects in $Top^0$.

\subsection{Cobar functor}\label{Subseccobar} Recall the definition of the cobar functor
$$\mathbf{\Omega}: dgCoalg^0_{\mathbf{k}} \to dgAlg_{\mathbf{k}}$$
where $dgAlg_{\mathbf{k}}$ denotes the category of unital dg algebras over $\mathbf{k}$.  For any connected dg coalgebra $(C, \partial :C \to C, \Delta:C\to C \otimes C)$ define a dg algebra
$$\mathbf{\Omega} C := ( T(s^{-1}  \tilde{C} ), D, \mu)$$
where $\tilde{C}_0:=0$ and $\tilde{C}_i:=C_i$ for $i>0$, $s^{-1}$ is the shift by $-1$ functor, $T(s^{-1} \tilde{C})= \bigoplus_{i=0}^{\infty} (s^{-1}\tilde{C})^{\otimes i}$ the tensor algebra with product $\mu$ given by concatenation of monomials, and the differential $D$ is defined as the unique extension of the linear map $$- s^{-1} \circ \partial \circ s^{+1} + (s^{-1} \otimes s^{-1}) \circ \Delta \circ s^{+1}: s^{-1}\tilde{C} \to s^{-1}\tilde{C} \oplus (s^{-1}\tilde{C} \otimes s^{-1}\tilde{C})$$ to a derivation. By doing this we obtain a map $D: T(s^{-1} \tilde{C}) \to T(s^{-1} \tilde{C})$. The coassociativity of $\Delta$ and the fact that $\partial^2 =0$ together imply that $D^2=0$. The unit is given by the inclusion $\mathbf{k} =  (s^{-1}\tilde{C})^{\otimes 0} \hookrightarrow T(s^{-1} \tilde{C})$. This construction is clearly functorial. The cobar functor may be defined more generally as a functor from coaugmented dg coalgebras to augmented dg coalgebras. 

\begin{definition} A map of connected dg coalgebras $f: C \to C'$ is called an $\mathbf{\Omega}$-\textit{quasi-isomorphism} if $\mathbf{\Omega}f: \mathbf{\Omega} C \to \mathbf{\Omega} C'$ is a quasi-isomorphism of dg algebras, i.e. if $\mathbf{\Omega}f$ induces an isomorphism on homology. 
\end{definition}

\begin{remark} Any $\mathbf{\Omega}$-quasi-isomorphism is a quasi-isomorphism of dg coalgebras but not vice versa. However, any quasi-isomorphism between simply connected dg coalgebras is an $\mathbf{\Omega}$-quasi-isomorphism. Recall that a dg coalgebra $C$ is said to be simply connected if $C_0 \cong \mathbf{k}$ and $C_1= 0$. For more details and a counterexample see section 2.4 in \cite{LoVa12}.
\end{remark}

\begin{definition} Let $\mathbf{C}$ and $\mathbf{C}'$ be connected $\rchi$-coalgebras with underlying connected dg coassociative coalgebras $C$ and $C'$, respectively. A map of connected $\rchi$-coalgebras $f: \mathbf{C} \to \mathbf{C'}$ is called an $\mathbf{\Omega}$-\textit{quasi-isomorphism} if the underlying map $f: C \to C'$ is an $\mathbf{\Omega}$-quasi-isomorphism in the sense of the previous definition. 
\end{definition}

\section{$\mathbf{\Omega}$-quasi-isomorphisms and the fundamental group}

In \cite{RiZe18} it was shown that the fundamental group of a path connected pointed space $X$ may be recovered from $\mathbf{C}_*(X)$. In this section we recall the precise statement and use it to show that if $f: X \to Y$ is a continuous map of path connected pointed spaces such that $\mathbf{C}_*(f): \mathbf{C}_*(X) \to \mathbf{C}_*(Y)$ is an $\mathbf{\Omega}$-quasi-isomorphism then $\pi_1(f): \pi_1(X) \to \pi_1(Y)$ is an isomorphism of groups. 

Let $X$ be a path connected topological space with a fixed base point $x \in X$ and denote by $\Omega X$ the topological monoid of (Moore) loops $\{\gamma: [0, r] \to X: r \in \mathbb{R}_{\geq 0}, \gamma(0)=\gamma(r)=x\}$ with the compact-open topology and associative product $\Omega X \times \Omega X \to \Omega X$ given by concatenating loops. Let $C_*(\Omega X)$ denote the dg algebra of normalized singular chains on $\Omega X$ with product induced by the product of $\Omega X$. Define $C^1_*(X)$ as the quotient dg coalgebra obtained by modding out the degenerate chains of the dg coalgebra generated by continuous maps $\sigma: \Delta^n \to X$ that send the $1$-skeleton of $\Delta^n$ to the base point $x \in X$ equipped with the usual boundary operator and coproduct given by the Alexander-Whitney diagonal approximation. Note that $C^1_0(X) \cong \mathbf{k}$ and $C^1_1(X) =0$, i.e. $C^1_*(X)$ is a simply connected dg coalgebra. We recall a classical result of Adams in which the cobar construction of $C^1_*(X)$ is compared to the dg algebra of chains on the based loop space of $X$.

\begin{theorem} \cite{Ad52} There exists a map of dg algebras $\theta: \mathbf{\Omega}C^1_*(X) \to C_*(\Omega X)$ which is a quasi-isomorphism if $\pi_1(X)=0$.
\end{theorem}
\begin{remark} The dg algebra map $\theta: \mathbf{\Omega}C^1_*(X) \to C_*(\Omega X)$ factors through the normalized singular cubical chains on $\Omega X$.
\end{remark}

The first and third author proved in Corollary 9.2 of \cite{RiZe16} an extension of Adams' Theorem to path connected, possibly non-simply connected, spaces. The result was reformulated in Theorem 1 of \cite{RiZe18} from a slightly different perspective. We restate it using the notation introduced in the preliminaries section and in a form which will be relevant for this article.

\begin{theorem} \cite{RiZe18} \label{thm1} 
The functor $C^0: Top^0 \to dgCoalg^0_{\mathbf{k}}$ satisfies the following properties:
\\
1) $C^0$ sends weak homotopy equivalences of spaces to $\mathbf{\Omega}$-quasi-isomorphisms of dg coalgebras, and
\\
2) for any $X \in Top^0$ there is a quasi-isomorphism of dg algebras $\mathbf{\Omega} C^0_*(X) \simeq C_*(\Omega X)$. 
\end{theorem}

In particular, it follows from the above statement that there are isomorphisms of algebras $H_0(\mathbf{\Omega} C^0_*(X) ) \cong H_0(\Omega X) \cong \mathbf{k} [ \pi_1(X) ]$. We describe explicitly the algebra isomorphism 
\begin{equation} \label{iso}
\psi: H_0(\mathbf{\Omega} C^0_*(X) )  \xrightarrow{\cong} \mathbf{k} [ \pi_1(X) ].
\end{equation} Note $(\mathbf{\Omega}C^0_*(X))_0$ is generated as a $\mathbf{k}$-module by monomials $[\sigma_1| ... |\sigma_n]$ where each $\sigma_i: \Delta^1 \to X \in C^0_1(X)$ is a loop based at $x \in X$. For any length $1$ monomial $[\sigma] \in (\mathbf{\Omega}C^0_*(X))_0$ we have $$\psi ( [\sigma] ) = \big[ \sigma \big] - e \in \mathbf{k} [\pi_1(X)],$$ where $\big[ \sigma \big] \in \pi_1(X)$ denotes the homotopy class of $\sigma$ and $e$ is the identity element of $\pi_1(X)$, namely the homotopy class of the constant loop at $x \in X$. Set $\psi( 1_{\mathbf{k}} ) = e$ and extend $\psi$ to monomials in $(\mathbf{\Omega}C^0_*(X))_0$ of arbitrary length as an algebra map. The resulting map $(\mathbf{\Omega}C^0_*(X))_0 \to \mathbf{k}[\pi_1(X)]$ induces the isomorphism $\psi: H_0(\mathbf{\Omega}C^0_*(X)) \xrightarrow{\cong}  \mathbf{k}[\pi_1(X)]$. 

The group algebra $\mathbf{k} [ \pi_1(X) ]$ is a cocommutative Hopf algebra when equipped with the coproduct determined by $\nabla(g) = g \otimes g$ for all $ g \in \pi_1(X)$. The antipode map is given by $s(g)= g^{-1}$. The counit $\epsilon: \mathbf{k} [ \pi_1(X) ] \to \mathbf{k}$  is given by defining $\epsilon(g)=1$  for all $g \in \pi_1(X)$. Hence, the isomorphism $H_0(\mathbf{\Omega} C^0_*(X) )  \cong \mathbf{k} [ \pi_1(X) ]$ determines a bialgebra structure on $H_0(\mathbf{\Omega} C^0_*(X) )$ which is in fact a Hopf algebra. The coproduct on $H_0(\mathbf{\Omega} C^0_*(X) )$ is given by the formula $\nabla ([\sigma])= [\sigma] \otimes [\sigma] + [\sigma] \otimes 1 + 1 \otimes [\sigma]$ for any $\sigma \in C^0_1(X)$. 

The bialgebra structure of $H_0(\mathbf{\Omega} C^0_*(X) )$ may be lifted to a dg bialgebra structure on $\mathbf{\Omega} C^0_*(X)$. The coproduct
\begin{equation}\label{Baues}
\nabla: \mathbf{\Omega} C^0_*(X) \to \mathbf{\Omega} C^0_*(X) \otimes \mathbf{\Omega} C^0_*(X)
\end{equation}
is deduced from the following observation:  $\mathbf{\Omega} C^0_*(X)$ is \textit{isomorphic} as a dg algebra to the normalized (cubical) chains on a monoidal cubical set with connections constructed in \cite{RiZe16} and denoted by $\mathfrak{C}_{\square_c}(\text{Sing}(X,x))$.  The complex of normalized chains on any cubical set with connections has a natural coproduct  making it a counital dg coalgebra which models the diagonal map at the level of geometric realizations; such a coproduct is also known as the Serre diagonal. It follows that $\mathbf{\Omega} C^0_*(X)$ becomes a dg bialgebra lifting the bialgebra structure of $H_0(\mathbf{\Omega} C^0_*(X) )$. An explicit formula for $\nabla$ on generators of positive total degree may also be found in \cite{Ba98}. 

Furthermore, it is explained in \cite{Ka03} how this dg bialgebra structure on $\mathbf{\Omega} C^0_*(X)$ may be deduced from certain maps corresponding to the co-action of the surjection operad $\rchi$ on $C^0_*(X)$ provided by Theorem \ref{BeFr} above. Kadeishvili identifies explicitly those surjections (those generators in the dg operad $\rchi$) that determine the coproduct $\nabla: \mathbf{\Omega}C \to \mathbf{\Omega}C \otimes \mathbf{\Omega}C$. We summarize exactly what we need for the purpose of this article in the following statement. 
\begin{proposition} \cite{Ka03}
Let $\mathbf{C}_*$ be a coalgebra over the surjection operad $\rchi$ and denote by $\mathbf{\Omega} C$ the cobar construction of the underlying dg coassociative coalgebra of $\mathbf{C}_*$. There exists a coproduct $\nabla: \mathbf{\Omega} C \to  \mathbf{\Omega} C \otimes  \mathbf{\Omega} C$ which makes $\mathbf{\Omega}C$ into a dg bialgebra and it agrees with the coproduct \ref{Baues} when $\mathbf{C}_*= \mathbf{C}_*(X)$.
\end{proposition} 

\begin{remark} It is shown in \cite{Fr10} that, if the $E_\infty$-cooperad we are working over is cofibrant, then the coproduct $\nabla: \mathbf{\Omega} C \to  \mathbf{\Omega} C \otimes  \mathbf{\Omega} C$ extends to an $E_{\infty}$-coalgebra structure which is compatible with the algebra structure of $\mathbf{\Omega}C$, a structure found in the literature under the name \textit{Hopf $E_{\infty}$-coalgebra}. In particular, this implies that $H_0(\mathbf{\Omega}C)$ is a bialgebra but not say anything about the existence of an antipode making it into a Hopf algebra. When $C$ is the dg coalgebra of singular chains we show additionally that the bialgebra $H_0(\mathbf{\Omega}C)$ has an antipode.
\end{remark}

Let $\mathcal{E}_{\mathbf{k}}^0$ be the full subcategory of $\rchi\text{-}Coalg^0$ whose objects are connected $\rchi$-coalgebras $\mathbf{C}$ which have the property that the bialgebra $H_0(\mathbf{\Omega} C)$ admits an antipode making it a Hopf algebra, where $C$ is the underlying dg coassociative coalgebra of $\mathbf{C}$. Since $H_0(\mathbf{\Omega} C^0_*(X) ) \cong \mathbf{k}[\pi_1(X)]$ and $C^0_*(X)$ is the underlying dg coalgebra of $\mathbf{C}_*(X)$, we may consider $\mathbf{C}_*(X)$ as an object of $\mathcal{E}_{\mathbf{k}}^0$. Therefore we will regard $\mathbf{C}_*: Top^0 \to \rchi\text{-}Coalg^0$ from now on as a functor 
$$\mathbf{C}_*: Top^0 \to \mathcal{E}_{\mathbf{k}}^0.$$
Let $F: \mathcal{E}_{\mathbf{k}}^0 \to dgCoalg^0_{\mathbf{k}}$ be the functor which sends a connected $\rchi$-coalgebra to its underlying connected dg coalgebra. Denote by $\mathcal{H}_{\mathbf{k}}$ the category of (unital, counital) Hopf algebras over $\mathbf{k}$ and by $\mathcal{G}$ the category of groups. Given a Hopf algebra $H$ with coproduct $\nabla: H \to H \otimes H$, antipode $s: H \to H$, unit $\eta: \mathbf{k} \to H$, and counit $\epsilon: H \to \mathbf{k}$, the group-like elements functor 
$$Grp: \mathcal{H}_{\mathbf{k}} \to \mathcal{G}$$
 is defined by
$$Grp (H) :=\{g \in H: \epsilon(g)= \eta(1_{\mathbf{k}}) \text{ and }\nabla(g)= g \otimes g \}.$$

We recall the main result of \cite{RiZe18}:

\begin{theorem} \label{pi1andchains}The functor $\mathbf{C}_*: Top^0 \to \mathcal{E}_{\mathbf{k}}^0$ satisfies
\\
1) $F \circ \mathbf{C}_*= C^0_*$
\\
2) $\mathbf{C}_*$ sends weak homotopy equivalences of pointed spaces to $\mathbf{\Omega}$-quasi-isomorphisms of $\rchi$-coalgbras, and 
\\
3) $Grp \circ H_0\circ \mathbf{\Omega} \circ F \circ \mathbf{C}_*= \pi_1.$
\end{theorem}
Note that in the above theorem we have combined Theorems 2 and 3 of \cite{RiZe18} into a single statement. We may now deduce the main result of this section. 
\begin{proposition} \label{fundgrpiso}
Suppose $f: X \to Y$ is a continuous map between path connected pointed topological spaces. If the induced map $\mathbf{C}_*(f): \mathbf{C}_*(X) \to \mathbf{C}_*(Y)$ is an $\mathbf{\Omega}$-quasi-isomorphism then $\pi_1(f): \pi_1(X) \to \pi_1(Y)$ is an isomorphism of groups.
\end{proposition}
\begin{proof} By definition, if the morphism $\mathbf{C}_*(f): \mathbf{C}_*(X) \to \mathbf{C}_*(Y)$ in $\mathcal{E}_{\mathbf{k}}^0$ is an $\Omega$-quasi-isomorphism, then $\mathbf{\Omega} C^0_*(f): \mathbf{\Omega} C^0_*(X) \to \mathbf{\Omega} C^0_*(Y)$ is a quasi-isomorphism of dg bialgebras. In particular,
\begin{equation} \label{iso2}
H_0 ( \mathbf{\Omega} C^0_*(f)): H_0 (\mathbf{\Omega} C^0_*(X) )\to H_0(\mathbf{\Omega} C^0_*(Y))
\end{equation}
is an isomorphism of Hopf algebras. Applying the functor $Grp$ to \ref{iso2} we get an isomorphism of groups
\begin{equation} \label{iso3}
Grp (H_0 ( \mathbf{\Omega} C^0_*(f))): Grp(H_0 (\mathbf{\Omega} C^0_*(X) )) \to Grp(H_0(\mathbf{\Omega} C^0_*(Y)))
\end{equation}
It follows from 3) in Theorem \ref{pi1andchains} and from the explicit description of the isomorphism \ref{iso} that the map \ref{iso3} is exactly $\pi_1(f): \pi_1(X) \to \pi_1(Y)$. 
\end{proof}

\section{$\mathbf{\Omega}$-quasi-isomorphisms and local coefficients}
Let $f: X \to Y$ be a continuous map of path connected pointed spaces such that $\mathbf{C}_*(f): \mathbf{C}_*(X) \to \mathbf{C}_*(Y)$ is an $\mathbf{\Omega}$-quasi-isomorphism. In the previous section we argued that $f$ induces an isomorphism $\pi_1(X) \cong \pi_1(Y)= \pi$. In this section we show that if $M$ is a left $\mathbf{k}[\pi]$-module then the induced map $C_*(f): C_*(X; M) \to C_*(Y; M)$ between singular chains with local coefficients is a quasi-isomorphism. We prove this by using the observation that $C_*(X; M)$ is quasi-isomorphic to the bar construction $B(\mathbf{\Omega} C^0_*(X), M)$. 

We define the chain complex $C_*(X; M)$ as a twisted tensor product \cite{Br59}. We recall the construction. Let  $(C, \partial_C, \Delta_C)$ be a dg coalgebra and consider the universal twisting cochain $\tau: C \rightarrowtail \tilde{C} \cong s^{-1} \tilde{C} \hookrightarrow \mathbf{\Omega} C$, where $\tilde{C}$ is defined in Section \ref{Subseccobar}. Note that $\tau$ is a linear map of degree $-1$ satisfying the Maurer-Cartan equation $d \tau + \tau \star \tau=0$ in the convolution dg algebra $(\text{Hom}(C, \mathbf{\Omega} C), d, \star)$. Given any left dg module $(M, \partial_M)$ over $\mathbf{\Omega} C$ with action denoted by $m: \mathbf{\Omega}C \otimes M \to M$ define a differential
$\partial_{\tau}: C \otimes M \to C \otimes M$
by
$$ \partial_{\tau} = \partial_C \otimes \text{id}_M + \text{id}_C \otimes \partial_M + ( \text{id} \otimes m) \circ (\text{id}_C \otimes \tau \otimes \text{id}_M) \circ (\Delta_C \otimes \text{id}_M)$$
It follows that $\partial_{\tau}^2=0$. The chain complex  $C \otimes_{\tau} M = (C \otimes M , \partial_{\tau})$ is called the twisted tensor product of $C$ and $M$ via $\tau$.  

For $X \in Top^0$, any left $\mathbf{k}[\pi_1(X)]$-module $M$ becomes a left dg $\mathbf{\Omega} C^0_*(X)$-module by regarding $M$ as a chain complex concentrated in degree $0$ with $0$ differential and defining an action $\mathbf{\Omega} C^0_*(X) \otimes M \to M$ via the projection map
\begin{equation} \label{rho}
q: \mathbf{\Omega} C^0_*(X) \rightarrowtail H_0(\mathbf{\Omega} C^0_*(X))\cong \mathbf{k}[\pi_1(X)].
\end{equation}
Define the chain complex $C_*(X;M)= C^0_*(X) \otimes_{\tau} M$ and call its homology $H_*(X;M)$ the homology of $X$ with local coefficients in $M$. The equivalence of our definition and the more classical definition as the homology of the ordinary tensor product $C_*(\tilde{X}) \otimes_{ \mathbf{k}[\pi_1(X)] } M$, where $\tilde{X}$ is the universal cover of $X$, follows from Theorem 5.8 of \cite{DaKi01} or may be deduced from the (more general) main result of \cite{Br59}. 

We now relate $C_*(X;M)$ to the bar construction. For any dg algebra $A$ with augmentation $\epsilon: A \to \mathbf{k}$ define a conilpotent dg coalgebra $(BA, D_{BA}, \Delta)$ as follows. 
\begin{equation*}
BA= \mathbf{k}1_{\mathbf{k}} \oplus \left(s \overline{A}\,\right) \oplus \left(s \overline{A}\,\right)^{\otimes 2} \oplus \left(s \overline{A}\,\right)^{\otimes 3} \oplus \cdots ,
\end{equation*}
where $\overline{A}= \text{ker} \epsilon$,  $s$ denotes the shift by $+1$ functor. We write monomials in $BA$ by $\{a_1 | ... | a_k\}$ where $a_i \in s\overline{A}$. Set $D_{BA}= -d_1 + d_2$ where
$$d_{1}\{ a_{1}|...| a_{n} \}   =\sum_{i=1}^{n} (-1)^{\epsilon_{i-1}}
\{ a_{1}|...|\partial_{A}a_{i}|...| a_{n}\},$$
$$d_{2} \{  a_{1}|...| a_{n}\}= \sum_{i=1}^{n-1} (-1)^{\epsilon_{i}}
\{ a_{1}|...| a_{i} a_{i+1} |...| a_{n}\},$$
where $\epsilon_i= |a_1| + ... |a_i| - i+1$. The associativity of the product in $A$ and the fact $\partial_A^2=0$ imply that $D_{BA}^2=0$. The coproduct $\Delta: BA \to BA \otimes BA$ is given by 
$$\Delta( \{ a_1 | ... | a_n \} ) = 1_{\mathbf{k}}\otimes \{a_1 |...|a_n\}+\{a_1|...|a_n\}\otimes 1_{\mathbf{k}}+ \sum_{i=1}^{n-1} \{a_1 | ... |a_i\} \otimes \{ a_{i+1} |...| a_n\}. $$
Given a left dg $A$-module $(M, \partial_M)$ with action $m: A \otimes M \to M$ define the one sided bar construction $B(A,M)$ to be the chain complex $(BA \otimes M, D_{B(A,M)})$ where
$$D_{B(A,M)}= D_{BA} \otimes \text{id}_M + \text{id}_M \otimes \partial_M + d$$
where $d( \{ a_{1}|...| a_{n} \} \otimes x ) =(-1)^{\epsilon_n} \{ a_{1}|...| a_{n-1} \} \otimes m(a_n \otimes x)$, where $\epsilon_n$ is the sign from above. 

\begin{proposition} If $X \in Top^0$ and $M$ is a left $\mathbf{k}[\pi_1(X)]$-module then there is a quasi-isomorphism of chain complexes $C_*(X,M) \xrightarrow{\simeq} B(\mathbf{\Omega} C^0_*(X), M )$, where $M$ is considered as a $\mathbf{\Omega} C^0_*(X)$-module via the projection map $q$ as in \ref{rho}.
\end{proposition}

\begin{proof} This follows from Proposition 5.1 of \cite{RiZe17} where for any conilpotent dg coalgebra $C$ and any left dg module $M$ over $\mathbf{\Omega} C$ an explicit quasi-isomorphism between $C \otimes_{\tau} M$ and $B(\mathbf{\Omega}C, M)$ is described. The quasi-isomorphism is given by $$\rho \otimes \text{id}_M: C \otimes_{\tau} M \to B \mathbf{\Omega} C \otimes M$$ where $\rho$ is the classical quasi-isomorphism of dg coalgebras $\rho: C \xrightarrow{\simeq}B \mathbf{\Omega} C$. The existence of such a quasi-isomorphism $C \otimes_{\tau} M \simeq B(\mathbf{\Omega}C, M)$ also follows from the following observation: both $B(\mathbf{\Omega}C, M)$ and $C \otimes_{\tau} M$ are chain complexes calculating the derived tensor product $\mathbf{k} \otimes^{\mathbb{L}}_{\mathbf{\Omega C} } M$, where $\mathbf{k}$ is considered as a right dg $\mathbf{\Omega}C$-module via the augmentation $\mathbf{\Omega} C \to \mathbf{k}$. $B(\mathbf{\Omega}C, M)$ is obtained by resolving $\mathbf{k}$ as a right dg $\Omega C$-module via the resolution $B(\mathbf{\Omega}C, \mathbf{\Omega}C) \to \mathbf{k}$ while $C \otimes_{\tau} M$ via the resolution $C \otimes \mathbf{\Omega}C \to \mathbf{k}$. 
\end{proof}

We now state and show the main result of this section.

\begin{proposition} \label{localcoeffs} Let $f: X \to Y$ be a continuous map of path connected pointed spaces such that $\mathbf{C}_*(f): \mathbf{C}_*(X) \to \mathbf{C}_*(Y)$ is an $\mathbf{\Omega}$-quasi-isomorphism. Suppose $M$ is a left $\mathbf{k}[\pi]$-module where $\pi= \pi_1(Y) \cong \pi_1(X)$. Then $C_*(f): C_*(X; M) \to C_*(Y; M)$ is a quasi-isomorphism.
\end{proposition}

\begin{proof} Since $\mathbf{C}_*(f)$ is an $\mathbf{\Omega}$-quasi-isomorphism then $\mathbf{\Omega} C^0_*(f) : \mathbf{\Omega}C^0_*(X) \to \mathbf{\Omega}C^0_*(Y)$ is a quasi-isomorphism of dg algebras. The map $\mathbf{\Omega} C^0_*(f)$ induces a quasi-isomorphism of chain complexes $B\mathbf{\Omega} C^0_*(f): B(\mathbf{\Omega} C^0_*(X), M ) \to B(\mathbf{\Omega} C^0_*(Y), M )$. This follows from the quasi-isomorphism invariance of the bar construction, a well known fact which can be proven by a standard spectral sequence argument as in Proposition 2.2.4 of \cite{LoVa12}. The result now follows from the commutative diagram

\[
\xymatrix{
C^0_*(X) \otimes_{\tau} M \ar[r]^-{\rho \otimes \text{id}_M} \ar[d]_{C^0_*(f) \otimes \text{id}_M }& B(\mathbf{\Omega} C^0_*(X), M) \ar[d]^{B\mathbf{\Omega} C^0_*(f)} \\
C^0_*(Y) \otimes_{\tau} M  \ar[r]^-{\rho \otimes \text{id}_M} & B(\mathbf{\Omega}C^0_*(Y),M).
}
\]

\end{proof}

\section{Whitehead's Theorem}

In this section we recall the homology version of Whitehead's Theorem which we will need to prove our main theorem. Whitehead's Theorem roughly states that a map of topological spaces is a weak equivalence if and only if  the map induces an isomorphism on fundamental groups and further induces an isomorphism on all homology groups with respect to a certain local system. This version of Whitehead's Theorem can be found as Theorem 6.71 in \cite{DaKi01}, because Davis and Kirk only include a sketch of the proof of this theorem, we include for completeness an explanation of how this theorem follows from the usual Whitehead Theorem. 

\begin{theorem}\label{whitehead}
Let $f:X \rightarrow Y$ be a continuous  map of connected pointed topological spaces. The map $f$ is a weak homotopy equivalence if and only if the following conditions hold
\begin{enumerate}
\item The induced map $f_*:\pi_1(X) \rightarrow \pi_1(Y)$ is an isomorphism.
\item The induced map $f_*:H_n(X;\Z[\pi]) \rightarrow H_n(Y;\Z[\pi])$ on homology with coefficients in the local system $\Z[\pi]$, where $\pi= \pi_1(Y) \cong \pi_1(X)$, is an isomorphism for all $n>0$. 
\end{enumerate} 
\end{theorem}

\begin{proof}
We prove the theorem by first passing to the universal covers of $X$ and $Y$ and then by using the simply connected homology version of Whitehead's Theorem.

Since we assumed that all the spaces in this paper are connected, locally path connected, and semi-locally simply connected the universal covers of $X$ and $Y$ exist. We denote these universal covers by $\tilde{X}$ and $\tilde{Y}$ and the covering projections by $p_X:\tilde{X} \rightarrow X$ and $p_Y:\tilde{Y} \rightarrow Y$. By the classical lifting criterion for covering spaces we may find a lift $\tilde{f}:\tilde{X} \rightarrow \tilde{Y}$ such that the following diagram commutes
\[
\xymatrix{
\tilde{X} \ar@{-->}[r]^{\tilde{f}} \ar[d]^{p_X}& \tilde{Y} \ar[d]^{p_Y} \\
X \ar[r]^{f} & Y.
}
\]

If  $f$ is a weak homotopy equivalence, then $f_*:\pi_1(X) \rightarrow \pi_1(Y)$ is an isomorphism. Since $f$ is a weak homotopy equivalence it also follows that the lift $\tilde{f}$ is a weak homotopy equivalence and so it induces an isomorphism $\tilde{f}_*: H_*(\tilde{X}) \rightarrow H_*(\tilde{Y})$. As is explained in Section 5.2 of \cite{DaKi01}, $H_n(X;\Z[\pi])$, the homology of the local system $\Z[\pi]$ is isomorphic to $H_n(\tilde{X})$, the homology of the universal cover. This fact may also be deduced from the main result of \cite{Br59}. Since the map $\tilde{f}$ is a weak equivalence, it induces an isomorphism in homology of the universal covers. The map $f$ therefore induces an isomorphism between $H_n(X;\Z[\pi])$ and $H_n(Y;\Z[\pi])$, which proves the first implication of the theorem.

To prove the converse, assume $f$ induces an isomorphism on fundamental groups and on homology with coefficients in the local system $\Z[\pi]$. Since $f$ induces an isomorphism on local homology, the induced map $\tilde{f}_*:H_*(\tilde{X}) \rightarrow H_*(\tilde{Y})$ is an isomorphism. We now use the simply connected version of Whitehead's Theorem (Corollary 6.69 in \cite{DaKi01}), which states that a map between simply connected spaces which induces an isomorphism in homology is a weak homotopy equivalence. This implies that $\tilde{f}$ is a weak homotopy equivalence so the universal covers $\tilde{X}$ and $\tilde{Y}$ are therefore weakly equivalent. The map $f$ therefore induces  isomorphisms  $f_*:\pi_n(X)\rightarrow \pi_n(Y)$ for all $n \geq 2$. Since we further assumed that the map $f_*:\pi_1(X) \rightarrow \pi_1(Y)$ is also an isomorphism, the map $f_*$ induces isomorphisms on all homotopy groups for $n \geq 1$ and is therefore a weak homotopy equivalence.

\end{proof}

\section{Main theorem}

We put all the pieces together to prove our main theorem. We now work over $\mathbf{k}= \mathbb{Z}$. Recall that $\mathbf{C}_*(X)$ denotes the connected $\rchi$-coalgebra of singular chains on a pointed path connected space $X$ and $C^0(X)$ denotes its underlying connected dg coassociative coalgebra with coproduct given by the Alexander-Whitney diagonal approximation. 

\begin{theorem} A continuous map of path connected pointed spaces $f: X \to Y$ is a weak homotopy equivalence if and only if $\mathbf{C}_*(f): \mathbf{C}_*(X) \to \mathbf{C}_*(Y)$ is an $\mathbf{\Omega}$-quasi-isomorphism of $\rchi$-coalgebras; namely, if the map of dg algebras $\mathbf{\Omega} C^0_*(f): \mathbf{\Omega} C^0_*(X) \to \mathbf{\Omega} C^0_*(Y)$ induces an isomorphism on homology. 
\end{theorem} 
\begin{proof} If $f: X \to Y$ is a weak homotopy equivalence then it follows from 1) and 2) of Theorem \ref{pi1andchains} that $\mathbf{C}_*(f): \mathbf{C}_*(X) \to \mathbf{C}_*(Y)$ is an $\mathbf{\Omega}$-quasi-isomorphism of $\rchi$-coalgebras.
Now for the converse assume that $f: X \to Y$ is a map such that $\mathbf{C}_*(f): \mathbf{C}_*(X) \to \mathbf{C}_*(Y)$ is an $\mathbf{\Omega}$-quasi-isomorphism of $\rchi$-coalgebras. By Proposition \ref{fundgrpiso}, it follows that $\pi_1(f): \pi_1(X) \to \pi_1(Y)= \pi$ is an isomorphism of groups. By Proposition \ref{localcoeffs}, it follows that $C_*(f): C_*(X;M) \to C_*(Y;M)$ is a quasi-isomorphism for any left $\mathbf{k} [\pi]$-module $M$. In particular, we have that for $M= \mathbf{\mathbb{Z} }[\pi]$, $C_*(f): C_*(X; \mathbb{Z}[\pi]) \to C_*(Y; \mathbb{Z}[\pi])$ is a quasi-isomorphism. Thus, $H_n(f): H_n(X; \mathbb{Z}[\pi]) \to H_n(Y; \mathbb{Z}[\pi])$ is an isomorphism for all $n>0$. It follows from Theorem \ref{whitehead} that $f: X \to Y$ is a weak homotopy equivalence.
\end{proof}

\bibliographystyle{plain}

 \Addresses

\end{document}